\documentclass[11pt, a4paper]{article}
\usepackage{amsmath, amssymb, amsfonts, amsthm}
\usepackage[pdftex]{graphicx, color}

\newcommand{\E}{{\bf{E}}}
\newcommand{\PP}{{\bf{P}}}
\newcommand{\Var}{{\bf{Var}}}
\newcommand{\Cov}{{\bf{Cov}}}

\newtheorem{pr}{Proposition}
\newtheorem{tm}{Theorem}
\newtheorem{lem}{Lemma}

\newtheorem{obs}{Observation}
\begin{document}

\parindent=0pt

\smallskip
\par\vskip 3.5em
\centerline{\Large \bf Connectivity threshold for superpositions  of Bernoulli}
\centerline{\Large \bf random graphs}

\vglue2truecm

\vglue1truecm
\centerline{Daumilas Ardickas and Mindaugas Bloznelis}

\bigskip

\centerline{Institute of Computer Science, Vilnius University}
\centerline{
\, \ Didlaukio 47, LT-08303 Vilnius, Lithuania} 

\vglue2truecm

\abstract{Let $G_1,\dots, G_m$ be independent
Bernoulli random subgraphs of the complete graph ${\cal K}_n$ having
variable sizes $x_1,\dots, x_m\in [n]$  and densities $q_1,\dots, q_m\in [0,1]$. 
Letting $n,m\to+\infty$, we study the connectivity threshold for the union 
$\cup_{i=1}^mG_i$ defined on the vertex set of ${\cal K}_n$. 
Assuming that the empirical distribution 
$P_{n,m}$
of the pairs $(x_1,q_1),\dots, (x_m,q_m)$ converges to a probability distribution $P$ we show that the threshold is defined by the mixed moments 
$\kappa_n=\iint  x(1-(1-q)^{|x-1|})P_{n,m}(dx,dq)$.
For $\ln n-\frac{m}{n}\kappa_n\to-\infty$ we have
 $\PP\{\cup_{i=1}^mG_i$ is connected$\}\to 1$ and for
$\ln n-\frac{m}{n}\kappa_n\to+\infty$
we have
 $\PP\{\cup_{i=1}^mG_i$ is connected$\}\to 0$.
Interestingly, this dichotomy only holds if
the mixed moment
$\iint  x(1-(1-q)^{|x-1|})\ln(1+x)P(dx,dq)<\infty$.
}
\smallskip
\par\vskip 3.5em
\section{Introduction}

Let $G_1,\dots, G_m$ be independent Bernoulli random subgraphs of the complete graph ${\cal K}_n$ having variable sizes $x_1,\dots, x_m$ and densities $q_1,\dots, q_m$. More precisely,
each $G_k=({\cal V}_k,{\cal E}_k)$ is obtained by  selecting 
a subset of vertices ${\cal V}_k$ of ${\cal K}_n$ of size 
$|{\cal V}_k|=x_k$ 
and retaining edges between selected vertices independently at random with probability $q_k$.  Therefore, each $G_k$  is a random graph on $x_k$ vertices, where every pair of vertices is linked by an edge independently at random with probability $q_k$.
The union  
$G_1\cup\cdots\cup G_m$
 has been used by  Yang and Leskovec \cite{Yang_Leskovec2012, Yang_Leskovec2014}
 as a benchmark network model in their study of overlapping community detection algorithms. Since then, this network model, called  community affiliation graph, has received considerable attention in the literature.
 
 In this paper we 
 study the connectivity property of the null model 
 of the community affiliation graph, where
 the vertex sets ${\cal V}_1$, $\dots$, ${\cal V}_m$  of the layers (communities) $G_1,\dots, G_m$ are drawn independently at random from the vertex set
  ${\cal V}=[n]$ of ${\cal K}_n$.
  Moreover, we consider the union 
  $G=\bigl({\cal V},\cup_{k=1}^m{\cal E}_k\bigr)$
  with the  (non-random) vertex set  ${\cal V}$ rather than 
  $\cup_{k=1}^m {\cal V}_k$.
 Note that 
 the probability distribution of the random graph  $G$
 is determined by the sequence of pairs $(x_1,q_1), \dots, (x_m,q_m)$
 or, equivalently, by the empirical distribution $P_{n,m}=\frac{1}{m}\sum_{k=1}^m\delta_{(x_k,q_k)}$ that assigns mass $\frac{1}{m}$
 to each member $(x_k,q_k)$ of the sequence.
 We let $n,m\to+\infty$ and assume that $P_{n,m}$ admits a limiting shape, that is, $P_{n,m}$ converges (weakly) to some probability distribution $P$ defined on $\{0,1,\dots\}\times[0,1]$.
 Let us mention that in the range of parameters  $m=m(n)=\Theta(n)$
 the random graph $G$ 
 admits a power law  asymptotic degree distribution with tunable exponent  and nontrivial clustering spectrum \cite{Lasse_Mindaugas2019}. Moreover, in the same range $m=\Theta(n)$  the giant connected component emerges that contains   a fraction of all vertices \cite{Lasse_Mindaugas2019}. In the present paper we address the property of full connectivity.
 In the range  of $m=\Theta(n\log n)$ we establish the connectivity threshold and  described it in terms of moments of $P$. 

The paper is organized as follows. In the remaining part of this section we  present our results and give a brief overview of related previous work. Proofs are postponed to the next section.

\subsection{Results}
We establish the connectivity threshold for a more general network model
admitting random community sizes, which we describe below.
Given $n$ and $m$, let $(X_{n,1},Q_{n,1}),\dots, (X_{n,m},Q_{n,m})$ be independent bivariate random variables taking values in 
$\{0,1,\dots, n\}\times [0,1]$.
For each $i\in[m]$, the $i$-th community 
$G_{n,i}=({\cal V}_{n,i},{\cal E}_{n,i})$ is defined as follows. 
We firstly sample $(X_{n,i},Q_{n,i})$. Then, given $(X_{n,i},Q_{n,i})$,
we select the vertex set ${\cal V}_{n,i}\subset [n]$ of size $|{\cal V}_{n,i}|=X_{n,i}$ uniformly at random and link every pair of elements of ${\cal V}_{n,i}$ independently at random with probability $Q_{n,i}$. In what follows we study   the union graph $G_{[n,m]}=({\cal V}, {\cal E})$. Here  
the vertex set ${\cal V}=[n]$ and the edge set 
${\cal E}=\cup_{i\in[m]}{\cal E}_{n,i}$. We call $G_{[n,m]}$ community affiliation graph 
defined by the sequence 
$(X_{n,1},Q_{n,1}),\dots, (X_{n,m},Q_{n,m})$.  In Theorem \ref{theorem} below we consider a sequence of random graphs $G_{[n,m]}$ for $n,m\ge 1$ and establish the zero - one law for the probability of  $G_{[n,m]}$ being connected as $n,m\to+\infty$. We will  assume that $m=m(n)$.

Before formulating the results we introduce some notation. Let $i_*$ be a random number uniformly distributed in $[m]$. We assume that $i_*$ is independent of the sequence $\{(X_{n,i},Q_{n,i})$, $i\in [m]\}$. The bivariate random variable 
$(X_{n,i_*},Q_{n,i_*})$ has probability distribution,
\begin{displaymath}
\PP\bigr\{X_{n,i_*}=x, Q_{n,i_*}\in [0,t]\bigl\}
=\frac{1}{m}\sum_{i\in[m]}\PP\{X_{n,i}=x, Q_{n,i}\in[0,t]\}.
\end{displaymath}
We denote this probability distribution $P_{n,m}$. We will assume that $P_{n,m}$ converges weakly to a probability distribution, say, $P$ on $\{0,1,2,\dots\}\times [0,1]$. 
Let $(X,Q)$ be a bivariate random variable with the distribution $P$.
Introduce the function 
\begin{displaymath}
h(x,q)=1-(1-q)^{(x-1)_+}
\end{displaymath}
 defined for $x\ge 0$ and $0\le q\le 1$.  
 Here 
 we denote $(a)_+=\max\{a,0\}$
and  assign value $1$ to the expression $0^0$. 
By ${\mathbb I}_{\cal A}$ we denote the indicator function of the condition (event) ${\cal A}$.
Denote 
\begin{displaymath}
\kappa_n=\E\bigl(X_{n,i_*}h(X_{n,i_*},Q_{n,i_*})\bigr)
=\frac{1}{m}\sum_{i\in[m]}\E\bigl(X_{n,i}h(X_{n,i},Q_{n,i})\bigr)
\end{displaymath}
and
\begin{displaymath}
\lambda_{n,m}=\ln n-\frac{m}{n}\kappa_n.
\end{displaymath}

\begin{tm}\label{theorem} Let $m,n\to+\infty$.
 Assume that   $(X_{n,i_*},Q_{n,i_*})$ converges in distribution to some
  bivariate random random variable $(X,Q)$ such that 
 $\E\bigl(Q{\mathbb I}_{\{X\ge 2\}}\bigr)>0$  and 
\begin{equation}
\label{2023-04-29+12}
\E \bigl(Xh(X,Q)\ln (1+X)\bigr)<\infty.
\end{equation} 
Assume, in addition, that
\begin{equation}
\label{2023-06-03}
\E \bigl(X_{n,i_*}h(X_{n,i_*},Q_{n,i_*})\ln (1+X_{n,i_*})\bigr)
\to
\E \bigl(Xh(X,Q)\ln (1+X)\bigr).
\end{equation} 

Then 
\begin{equation}
\label{2023-05-08}
\PP\{G_{[n,m]}\ {\text{\rm{is connected}}}\}
\to 
\begin{cases} 0,
\quad
{\text{\rm{for}}}
\quad
\lambda_{n,m}\to +\infty,
\\
1,
\quad
{\text{\rm{for}}}
\quad
\lambda_{n,m}\to -\infty.
\end{cases}
\end{equation} 
\end{tm}
We note that Theorem \ref{theorem} covers the case of 
non-random community sizes and densities. 
Let $\left((x_{n,1},q_{n,1}),\dots, (x_{n,m}, q_{n,m})\right)$,
$n=1,2,\dots$ with $m=m(n)$ 
 be a sequence  of non-random  bivariate vectors 
such that $(x_{n,i},q_{n,i})\in\{0,1,\dots, n\}\times [0,1]$ for each 
$n$ and $1\le i\le m$. 
To treat this case one applies Theorem \ref{theorem} 
to the   series  of degenerate random vectors
$\bigl((X_{n,1},Q_{n,1}),\dots,(X_{n,m},Q_{n,m})\bigr)$, $n=1,2,\dots$ with $m=m(n)$, where
$\PP\bigl\{(X_{n,i},Q_{n,i})=(x_{n,i},q_{n,i})\bigr\}=1$
for $1\le i\le m$, $n\ge 1$.

Another interesting case is when the community sizes are identically distributed. Let $(X,Q)$, $(X_1,Q_1)$, $(X_2,Q_2),\dots$ be independent identically distributed bivariate random variables taking values in 
$\{0,1,\dots\}\times [0,1]$. In Proposition \ref{proposition} below we consider community affiliation graph defined by the sequence $({\tilde X}_{n,1},Q_{n,1}),\dots, ({\tilde X}_{n,m},Q_{n,m})$, where 
${\tilde X}_{n,i}:=\min\{X_i,n\}$. In this case in (\ref{2023-05-08}) we can replace $\lambda_{n,m}$  by 
\begin{displaymath}
\lambda'_{n,m}=\ln n-\frac{m}{n}\kappa,
\qquad
{\text{where}}
\qquad 
\kappa=\E\bigl(Xh(X,Q)\bigr).
\end{displaymath}

\begin{pr}\label{proposition}
Let $m,n\to+\infty$.
 Assume that  
 $\E\bigl(Q{\mathbb I}_{\{X\ge 2\}}\bigr)>0$  and (\ref{2023-04-29+12}) holds. Then 
 \begin{equation}
\label{A}
\PP\{G_{[n,m]}\ {\text{\rm{is connected}}}\}
\to 
\begin{cases} 0,
\quad
{\text{\rm{for}}}
\quad
\lambda'_{n,m}\to +\infty,
\\
1,
\quad
{\text{\rm{for}}}
\quad
\lambda'_{n,m}\to -\infty.
\end{cases}
\end{equation} 
\end{pr}

Let us discuss the optimality of the moment conditions
 $\E\bigl(Q{\mathbb I}_{\{X\ge 2\}}\bigr)>0$  and 
 (\ref{2023-04-29+12}).
We notice that condition $\E\bigl(Q{\mathbb I}_{\{X\ge 2\}}\bigr)>0$ of Proposition \ref{proposition}
is neccessary as the opposite assumption that
 $\E\bigl(Q{\mathbb I}_{\{X\ge 2\}}\bigr)=0$ implies
that each $G_{n,i}$ has no edges almost surely. Consequently,  $G_{[n,m]}$ is an empty graph and (\ref{2023-05-08}) fails.
On the other hand, the moment condition (\ref{2023-04-29+12})  may look too restrictive
as it involves an extra logarithmic factor 
compared to the moment $\kappa=\E (Xh(X,Q))$ that describes the threshold
 (\ref{2023-05-08}). Interestingly, the logarithmic factor in (\ref{2023-04-29+12}) can't be 
 removed as  shown in Example below.

\medskip

{\bf Example}. Let  $f$ be a positive function  defined on the set of natural numbers ${\mathbb N}$ satisfying
$\lim_{x\to+\infty}f(x)=0$.
 Choose an increasing sequence 
$\{y_k, k\ge 1\}\subset {\mathbb N}$  such that $\sum_kf(y_k)<\infty$
and put $x_k=2^{2^{y_k}}$ and $p_k=(x_k\ln x_k)^{-1}$.
Denote $S=\sum_{k}p_k$ and 
define the probability distribution $P_{f}$  on ${\mathbb N}$:
$P_{f}(x)=\frac{1}{S}\sum_kp_k\delta_{x_k}$, $x\in {\mathbb N}$. Here $\delta_{x_k}$ denotes the Dirac measure assigning mass $1$ to $x_k$. 
Let $X'$ be a random variable with the distribution $P_{f}$.
It is easy to verify that
\begin{equation}
\label{2023-05-08+1}
\E X'<\infty,
\qquad
\E (X'f(X')\ln(1+X'))<\infty
\qquad
{\text{\rm{and}}}
\qquad
\E (X'\ln(1+X'))=\infty.
\end{equation}
Choose $L>0$ such that 
$\E \bigl({\mathbb I}_{\{X'\ge L\}}X'\bigr)\le 0.5$
and put $X=X'{\mathbb I}_{\{X'\ge L\}}$. 
Clearly, (\ref{2023-05-08+1}) holds for $X$ as well. 
Let $X_1,X_2,\dots$ be independent copies of $X$.
Let $G_{[n,m]}$ be community affiliation graph defined by the sequence of bivariate random vectors $({\tilde X}_{n,i},1)$, $1\le i\le m$. (Here 
we consider the particular case where each $Q_{n,i}\equiv 1$.)
For $m=\lfloor n\ln n\rfloor$ we have $\lambda_{n,m}\ge 0.5\ln n\to+\infty$.
On the other hand,
\begin{align*}
\PP\{G_{[n,m]}\ {\text{\rm{is connected}}}\}
&
\ge
\PP\{\max_{1\le i\le m}X_i\ge n\}
=
1-\PP\{\max_{1\le i\le m}X_i< n\}
\\
&
=
1-\left(\PP\{X< n\}\right)^m
=
1-\left(1-\PP\{X\ge n\}\right)^m
\\
&
\ge
1-e^{-m\PP\{X\ge n\}}.
\end{align*}
Invoking the inequality $\PP\{X\ge n\}\ge \PP\{X=n\}=(n\ln n)^{-1}$,
which holds for $n=x_k$ satisfying $n>L$, we obtain that 
$\PP\{G_{[n,m]}\ {\text{\rm{is connected}}}\}\ge 1-e^{-1}$.
Hence (\ref{2023-05-08}) fails.

\subsection{Related previous work}

Since the seminal work of Erd\H os and R\'enyi \cite{ErdosRenyi1959}  connectivity thresholds for various  random graph models remains an area of active research, see \cite{FriezeKaronski},
\cite{JansonLuczakRucinski2001}, \cite{Hofstad2017} and references therein.
 Erd\H os and R\'enyi \cite{ErdosRenyi1959} considered the random  graph  (now commonly called Erd\H os-R\'enyi graph) 
on $n$ vertices and with $m$ edges selected uniformly at random from the set of $\binom{n}{2}$ possible ones.
They identified and described the connectivity threshold in the parametric range 
$m=\Theta \left(\frac{\ln n}{n}\right)$ as $n,m\to+\infty$.
They showed, in particular, that the probability of  Erd\H os-R\'enyi graph being connected tends to $1$ for $\ln n-2\frac{m}{n}\to-\infty$ and
it tends to $0$ for $\ln n-2\frac{m}{n}\to+\infty$.
This result resembles that of Proposition \ref{proposition} above when applied to the degenerate bivariate random vector $(X,Q)\equiv (2,1)$ (in this case each layer $G_{n,i}$ consist of  a single edge).
Now the only difference between the  two models is that the edges of the Erd\H os-R\'enyi graph are all distinct, while the union  of overlaping layers
$\cup_{i\in [m] }G_{n,i}$ considered in Proposition \ref{proposition} may contain identical layers (complete overlap). Since the probability of the complete overlap is rather small, the connectivity threshold for both models 
is
the same.

One important property of the random graph model considered in the present paper is that it admits inhomogeneous degree sequences. For example, in the (sparse) parametric regime $m=\Theta(n)$ one can obtain the asymptotic power-law degree distribution with a  tunable power-law exponent including  distributions  with infinite second moment \cite{Lasse_Mindaugas2019}. Interestingly, despite the  highly variable degree sequence,  the only characteristic
 that shows up in the threshold description (\ref{2023-05-08}) is the mixed moment $\kappa_n$ (or its limit $\kappa$ in (\ref{A})).
  This is not the case for 
  another class of inhomogeneous random graphs, where edges between vertices, say $u,v$, are inserted independently at random with  probability, say $p_n(u,v)$, depending on each vertex pair 
  $\{u,v\}$. The description of the connectivity threshold for inhomogeneous random graph is a bit more complex  and involves 
  $\operatorname*{ess\,inf}$ of the properly scaled limit of $p_n$ as $n\to+\infty$, see \cite{DevroyeFraiman}. 

Another important feature of our random graph model is that it has the clustering property.
 In particular, it admits non-vanishing global clustering coefficient \cite{ValentasMindaugas_2016}, \cite{Lasse_Mindaugas2019}, see also \cite{PettiVempala2022}.
The clustering property is implied by the bipartite structure defined by the community memberships 
see, e.g.,  \cite{BGJKR2015properties},
\cite{Guillaume+L2004},
\cite{Kurauskas2022}, \cite{Hofstad_Komjathy_Vadon2021}. 
For the random graph models, where actors ($=$vertices) select their community memberships at random and independently across the actors,
and where all actors belonging to a given community are mutually adjacent,
 the  connectivity threshold was established in \cite{BlackburnGerke2009}, 
\cite{Rybarczyk2011},
\cite{Yagan2016},
\cite{YaganMakowski2012}.
For the random graph models, where communities co-opt their 
members at random  and independently of the other communities, 
and where all actors sharing a community are mutually adjacent,
 the  connectivity  threshold was established in
 \cite{BergmanLeskela} and \cite{GodehardJaworskiRybarczyk2007}.
 This latter random graph model corresponds to the particular 
 instance of community affiliation graph $G_{[n,m]}$ where we set 
 all the community edge densities to be non-random and equal to $1$, i.e., 
 $Q_{n,i}\equiv 1$ $\forall n,i$.
In this case
$G_{[n,m]}=({\cal V}, \cup_{i\in [m]}{\cal E}_{n,i})$ is a union of cliques of variable sizes $X_{n,1},\dots, X_{n,m}$. In the literature this random graph  is sometimes called the ``passive'' random intersection graph \cite{GodehardJaworski2001}.
In \cite{GodehardJaworskiRybarczyk2007} the connectivity threshold (\ref{2023-05-08}) has been established  for passive random intersection graph
with cliques having  uniformly bounded sizes ($\exists M>0$ such that
$\PP\{X_{n,i}\le M\}=1$ $\forall n\ge 1$ and $1\le i\le m$). The result of  \cite{GodehardJaworskiRybarczyk2007} has been extended by \cite{BergmanLeskela} to cliques having square integrable sizes.
Our Theorem \ref{theorem} gives a further extension of the results
of \cite{BergmanLeskela} and \cite{GodehardJaworskiRybarczyk2007}: we introduce the community thinning (egdes of  the $i$-th clique of size $X_i$ are retained independently at random with probability $Q_i$) and relax the moment condition up to the minimal one, see (\ref{2023-04-29+12}) above.

\bigskip
 
 We conclude the section with a brief remark about the future work. 
 The zero-one law of Theorem \ref{theorem} could be refined by 
 providing the asymptotics for the probability of $G_{[n,m]}$ being 
 connected in the parametric regime $\lambda_{n,m}\to c$, where $c>0$ is a constant.  A more interesting question is about the connectivity threshold for the union $G'_{[n,m]}=(\cup_{i\in [m]}{\cal V}_{n,i},\cup_{i\in [m]}{\cal E}_{n,i})$, where  the  vertex set 
$ \cup_{i\in [m]}{\cal V}_{n,i}$ is random.

\section{Proofs}

Here we prove Theorem \ref{theorem} and Proposition \ref{proposition}.
We give the proof of Proposition \ref{proposition} first as it is less technical than (and still contains the main ingredients of) the proof of our main result, Theorem \ref{theorem}.

\medskip

Before the proof of Proposition \ref{proposition} we introduce some notation and collect auxiliary results.
Let $A_x$ be a  subset of $[n]$ sampled uniformly at random from the class of subsets of size $x$. Let $q\in [0,1]$. Let $G_x=(A_x,{\cal E}_x)$ be  Bernoulli random graph with the vertex set $A_x$ and with egde probability $q$. Let
\begin{displaymath}
q_r(x,q)=\PP\bigl\{\forall \{u,v\}\in{\cal E}_x
\
\
{\text{ either}}
\
\
 u,v\in[r]
 \
 \
 {\text{or}}
 \
 \
 u,v\in[n]\setminus [r]
 \bigr\}
\end{displaymath}
denote the probability
that none of the edges of $G_x$ connect subsets 
$[r]$ 
and 
$[n]\setminus[r]$
of 
${\cal V}=[n]$.  The probability 
that none of the edges of $G_{n,k}$ connect subsets 
$[r]$ 
and 
$[n]\setminus[r]$
 is denoted 
\begin{displaymath}
{\bar q}_r=\E q_r({\tilde X}_k,Q_k).
\end{displaymath}
  Note that 
  $q_r(x,q)=1$ for $x=0,1$. Hence 
   ${\bar q}_r
   =
   \PP\{X\le 1\}+\E \bigl(
   q_r({\tilde X}_k,Q_k){\mathbb I}_{\{{\tilde X}_k\ge 2\}}
   \bigr)$.
We also note that 
neither $q_r(x,q)$ nor ${\bar q}_r$ depend on $k$.

\begin{lem}\label{Daumilo-lema}
For $1\le r\le n/2$, $2\le x\le n$ and $0\le q\le 1$ we have
\begin{align}
\label{2023-05-03}
q_r(x,q)
&
\le 
1-2\frac{r(n-r)}{n(n-1)}q,
\\
\label{2023-05-01+1}
q_r(x,q)
&
\le 
1-\left(\frac{r}{n}x-R_1-R_2\right)h(x,q),
\end{align}
where $R_1=\frac{r^2}{(n-r)^2}$ 
and 
$R_2=R_2(r,x)=e^{-\frac{r}{n}x}-1+\frac{r}{n}x$.
\end{lem}

\begin{proof}[Proof of Lemma \ref{Daumilo-lema}]
We have
\begin{align}
\nonumber
q_r(x,q)
&
=
{\binom{n}{x}}^{-1}
\sum_{j=0}^x
{\binom{r}{j}}
{\binom{n-r}{x-j}}
(1-q)^{j(x-j)}
\\
\nonumber
&
\le
{\binom{n}{x}}^{-1}
\left(
\sum_{j=0}^x
{\binom{r}{j}}
{\binom{n-r}{x-j}}
(1-q)^{x-1}
+
\left(
{\binom{r}{x}}
+
{\binom{n-r}{x}}
\right)
(1-(1-q)^{x-1})
\right)
\\
\label{2023-05+4}
&
=
(1-q)^{x-1}
+
{\binom{n}{x}}^{-1}
\left(
{\binom{r}{x}}
+
{\binom{n-r}{x}}
\right)
(1-(1-q)^{x-1})
.
\end{align}
Proof of (\ref{2023-05-01+1}).  We estimate the second term of 
(\ref{2023-05+4}). We have 
\begin{align*}
&
A:={\binom{n}{x}}^{-1}{\binom{n-r}{x}}=\frac{(n-r)_x}{(n)_x}\le 
\left(1-\frac{r}{n}\right)^x
\le e^{-\frac{r}{n}x},
\\
&
B:={\binom{n-r}{x}}^{-1}{\binom{r}{x}}
=\frac{(r)_x}{(n-r)_x}
\le 
\left(\frac{r}{n-r}\right)^x
\le 
\left(\frac{r}{n-r}\right)^2.
\end{align*}
Combining these estimates and using simple inequalities $A\le 1$ and $B\le 1$ we obtain that
\begin{align*}
{\binom{n}{x}}^{-1}
\left(
{\binom{r}{x}}
+
{\binom{n-r}{x}}
\right)
=A(B+1)
\le 
B+A
\le 
1-\frac{r}{n}x+R_1+R_2.
\end{align*}
Invoking the latter inequality in (\ref{2023-05+4}) we arrive to 
(\ref{2023-05-01+1}).

Proof of (\ref{2023-05-03}).  We have
\begin{align*}
q_r(x,q)
\le
q_r(2,q)
\le 
1-q+
{\binom{n}{2}}^{-1}
\left(
\binom{r}{2}
+
\binom{n-r}{2}
\right)
{\color{red}q}
=
1
-
2
\frac{r(n-r)}{n(n-1)}
q.
\end{align*}
Here the second inequality follows from 
(\ref{2023-05+4}). To verify the first inequality we sample a subset $U_x\subset [n]$ of size $x$ uniformly at random and, given $U_x$, we sample a subset $U_2\subset U_x$ of size $2$ uniformly at random. 
Given $U_x$, we generate a  Bernoulli random graph (denoted $G_x$) on the vertex set $U_x$ with the edge density $q$.
By $G_2$ we denote its subgraph induced by $U_2$.
Now $1-q_r(x,q)$ (respectively $1-q_r(2,q)$) is the probability that $G_x$ (respectively $G_2$) contains an edge with one endpoint in $[r]$ and another one in $[n]\setminus[r]$. The obvious coupling $\PP\{G_2\subset G_x\}=1$
implies the inequality $1-q_r(x,q)\ge 1-q_r(2,q)$.
Hence $q_r(x,q)\le q_r(2,q)$.
\end{proof}

Now we introduce
two auxiliary functions $\phi(x)$ and $\varphi(x)$ 
related to the distribution of $(X,Q)$.
The integrability condition
(\ref{2023-04-29+12}) implies that
\begin{displaymath}
\E\bigl( h(X,Q)X\ln(1+X)\bigr)
=\sum_{i\ge 2}i\ln (1+i)\E (h(X,Q)|X=i)\PP\{X=i\}
<\infty.
\end{displaymath}
In particular, one can find a positive increasing function $\phi(x)$ such that 
$\lim_{x\to+\infty}\phi(x)=+\infty$ and the series  
$\sum_{i\ge 2}i\ln (1+i)\E (h(X,Q)|X=i)\PP\{X=i)\phi(i)$ converges.
Hence
\begin{equation}
\label{2023-06-13}
\E\bigl( h(X,Q)\phi(X)X\ln(1+X)\bigr)<\infty.
\end{equation}

Moreover,  one can choose $\phi(x)$ such that
 the function 
$x\to x/(\phi(x)\ln (1+x))$ 
was increasing for $x=1,2,\dots$.
Furthermore, the integrability property (\ref{2023-06-13})
implies that the function
\begin{displaymath}
\varphi(t):=
\E\left( h(X,Q)\phi(X)X\ln(1+X){\mathbb I}_{\{X\ge t\}}\right)
\end{displaymath}
decays to zero as $t\to+\infty$.

\begin{proof}[Proof of Proposition \ref{proposition}]
Given $n$,
we write, for short, 
${\tilde X}=\min\{n, X\}$ and ${\tilde X}_{k}={\tilde X}_{n,k}$. 
We denote
${\tilde \kappa}={\tilde \kappa}_n
=\E\bigl({\tilde X}h({\tilde X},Q)\bigr)$  and
${\tilde\lambda}_{n,m}=\ln n-\frac{m}{n}{\tilde \kappa}_n$

The proof consists of two steps. In the first step we show that
\begin{equation}
\label{2023-05-08+4}
\PP\{G_{[n,m]}\ {\text{\rm{is connected}}}\}
\to 
\begin{cases} 0,
\quad
{\text{\rm{for}}}
\quad
{\tilde \lambda}_{n,m}\to +\infty,
\\
1,
\quad
{\text{\rm{for}}}
\quad
{\tilde \lambda}_{n,m}\to -\infty.
\end{cases}
\end{equation} 
In the second step we derive (\ref{A}) from (\ref{2023-05-08+4}).

{\it Proof of   (\ref{2023-05-08+4}).}
 We firstly consider the case where
 ${\tilde\lambda}_{n,m}\to+\infty$. Note that  now
$m=(\ln n-{\tilde\lambda}_{n,m}) n/{\tilde \kappa}_n=O(n\ln n)$. 
Let us prove that
 \begin{equation}
\label{2023-04-27}
\PP\{G_{[n,m]}\ {\text{\rm{is connected}}}\}\to 0.
\end{equation} 
Before the proof we introduce some notation. Given $v,u\in {\cal V}$ we introduce events 
${\cal I}_v=\{v$ is an isolated vertex of $G_{[n,m]}\}$ and
${\cal I}_{u,v}={\cal I}_v\cap{\cal I}_u$.  
For $1\le k\le m$ we introduce events
${\cal A}_{v|k}=\{v\in {\cal V}_{n,k}\}$
and
${\cal I}_{v|k}=\{ v\in {\cal V}_{n,k}$ and $v$  is an isolated vertex of $G_{n,k}\}$.
Given an event ${\cal A}$ we denote by ${\bar {\cal A}}$ and 
${\mathbb I}_{\cal A}$ the complement  event and  the indicator function, respectively.

Now we are ready to prove  (\ref{2023-04-27}). Let $N_0$ denote the number of isolated vertices 
of $G_{[n,m]}$.
We have 
$N_0
=
\sum_{v\in {\cal V}}{\mathbb I}_{{\cal I}_v}$.
Clearly, the inequality $N_0\ge 1$ implies that 
$G_{[n,m]}$ is disconnected. 
Therefore 
(\ref{2023-04-27}) would follow if we show 
that $\PP\{N_0=0\}\to 0$. To 
show the latter relation we  estimate (using  Chebyshev's 
inequality)
\begin{displaymath}
\PP\{N_0=0\}
\le 
\PP\{|N_0-\E N_0|\ge \E N_0\}
\le 
(\E N_0)^{-2}\Var N_0
\end{displaymath}
and invoke the bound $\Var N_0=o(\E N_0)^2$, see below. 
Hence it remains to  show that 
\begin{align}
\label{2023-04-28}
&
\E N_0
=  
ne^{m\ln(1-n^{-1}{\tilde \kappa})}
=
e^{{\tilde\lambda}_{n,m}+O(n^{-2}m{\tilde \kappa}^2)}
=
e^{{\tilde\lambda}_{n,m}+o(1)},
\\
\label{2023-04-28+1}
&
\Var N_0=o(\E N_0)^2.
\end{align}

Proof of (\ref{2023-04-28}).
Fix $u\in {\cal V}$.  We have, by symmetry, 
$\E N_0=\sum_{v\in {\cal V}}
\PP\{{\cal I}_{v}\}=n\PP\{{\cal I}_{u}\}$.
To evaluate the probability $\PP\{{\cal I}_{u}\}$ we observe that event ${\cal I}_{u}$ occurs when for each layer $G_{n,k}$ we have that either
$u$ is an isolated vertex of $G_{n,k}$ or $u$ is not included in the vertex set of $G_{n,k}$.  For each $k$ this happens with probability
\begin{displaymath}
p_k=\PP\{{\tilde X}_k\le 1\}
+
\PP\{{\bar {\cal A}}_{u|k}\cap \{{\tilde X}_k\ge 2\}\}
+
\PP\{{\cal I}_{u|k}\cap \{{\tilde X}_k\ge 2\}\}.
\end{displaymath}
Furthermore, by the independence and the  equality of distributions of $G_{n,1},\dots, G_{n,m}$,  we have
\begin{align}
\label{2023-04-28+2}
\PP\{{\cal I}_u\}
=
\prod_{k\in [m]}
p_k
=p_1^m.
\end{align}
Now we evaluate $p_k$.
A simple calculation shows that
\begin{align*}
\PP\{
{\bar {\cal A}}_{u|k}
\cap
\{{\tilde X}_k\ge 2\}\}
&
=
\E
\Bigl(
\PP\{{\bar {\cal A}}_{u|k}|X_k\}
{\mathbb I}_{\{{\tilde X}_k\ge 2\}}
\Bigr)
\\
&
=
\E\left((1-n^{-1}{\tilde X}_k){\mathbb I}_{\{{\tilde X}_k\ge 2\}}\right)
=
\PP\{{\tilde X}_k\ge 2\}-n^{-1}\E({\tilde X}_k{\mathbb I}_{\{{\tilde X}_k\ge 2\}}),
\\
\PP\{{\cal I}_{u|k}\cap \{{\tilde X}_k\ge 2\}\}
&=
\E
\Bigl(
\PP\{{\cal I}_{u|k}|X_k\}{\mathbb I}_{\{{\tilde X}_k\ge 2\}}
\Bigr)
=
\E\Bigl(n^{-1}{\tilde X}_k(1-Q_k)^{{\tilde X}_k-1}
{\mathbb I}_{\{{\tilde X}_k\ge 2\}}
\Bigr).
\end{align*}
Invoking these identities in the expression for $p_k$ above we obtain
\begin{equation}
\label{2023-05-13+20}
p_k=1-\E\Bigl(n^{-1}{\tilde X}_kh({\tilde X}_k,Q_k)
{\mathbb I}_{\{{\tilde X}_k\ge 2\}}
\Bigr)
=
1-n^{-1}{\tilde \kappa}.
\end{equation}
Finally, we write (\ref{2023-04-28+2}) in the form 
\begin{equation}
\label{2023-04-29+10}
\PP\{{\cal I}_u\}=e^{m\ln(1-n^{-1}{\tilde \kappa})}
=
e^{-\frac{m}{n}{\tilde \kappa}+O(n^{-2}m{\tilde \kappa}^2)}
\end{equation}
and  arrive to (\ref{2023-04-28}). In the last step we used 
${\color{red}\ln(1-t)= -t+O(t^2)}$ as $t\to 0$.

Proof of (\ref{2023-04-28+1}).
Fix $u,v\in {\cal V}$. We have, by symmetry,
\begin{align*}
\Var N_0
&
=\E N_0^2-(\E N_0)^2
=n \Var {\mathbb I}_{{\cal I}_u}
+n(n-1)
\Cov ({\mathbb I}_{{\cal I}_u},{\mathbb I}_{{\cal I}_v}).
\end{align*}
Note that
\begin{displaymath}
n\Var {\mathbb I}_{{\cal I}_u}
\le 
n
\PP\{{\cal I}_u\}
=
\E N_0
=
o\bigl((\E N_0)^2\bigr),
\end{displaymath}
since $\E N_0\to+\infty$.
It remains to show that
\begin{equation}
\label{2023-04-29+11}
n(n-1)
\Cov ({\mathbb I}_{{\cal I}_u},{\mathbb I}_{{\cal I}_v})
=
o\bigl((\E N_0)^2\bigr).
\end{equation}
Here $
\Cov ({\mathbb I}_{{\cal I}_u},{\mathbb I}_{{\cal I}_v})
=\PP\{{\cal I}_{u,v}\}-\PP\{{\cal I}_{u}\}\PP\{{\cal I}_{v}\}$. Let us  evaluate the  probability $\PP\{{\cal I}_{u,v}\}$.
By the independence and the equality of distributions of $G_{n,1},\dots, G_{n,m}$, we have that
\begin{equation}
\label{2023-04-28+5}
\PP\{{\cal I}_{u,v}\}
=
\prod_{k=1}^mq_k=q_1^m,
\end{equation}
where $q_k$ is the probability that either ${\tilde X}_k\le 1$
or ${\tilde X}_k\ge 2$ and one of the following alternatives hold:

${\cal B}_1$:  $\{u,v\}\cap {\cal V}_{n,k}=\emptyset$;

${\cal B}_2$: $\{u,v\}\cap {\cal V}_{n,k}=\{u\}$ and $u$ is 
isolated in $G_{n,k}$;

${\cal B}_3$: $\{u,v\}\cap {\cal V}_{n,k}=\{v\}$ and $v$ is 
isolated in $G_{n,k}$;

${\cal B}_4$: $\{u,v\}\subset {\cal V}_{n,k}$ and both $u$ and $v$ are 
isolated in $G_{n,k}$.

Hence, 
$q_k
=
\PP\{{\tilde X}_k\le 1\}
+
\sum_i\PP\{{\cal B}_i\cap \{{\tilde X}_k\ge 2\}\}$. 

Let us evaluate  probabilities 
$\eta_{k,i}=\PP\{{\cal B}_i\cap \{{\tilde X}_k\ge 2\}\}$.
We have
\begin{align*}
\eta_{k,1}
&=
\E\left(\PP\{{\cal B}_1|X_k){\mathbb I}_{\{{\tilde X}_k\ge 2\}}\right)
=
\E\left(\frac{(n-{\tilde X}_k)_2}{(n)_2}
{\mathbb I}_{\{{\tilde X}_k\ge 2\}}\right),
\\
\eta_{k,3}
&
=
\eta_{k,2}
=
\E\left(\PP\{{\cal B}_2|X_k,Q_k){\mathbb I}_{\{{\tilde X}_k\ge 2\}}\right)
=
\E\left(
\frac{{\tilde X}_k (n-{\tilde X}_k)}{(n)_2}
(1-Q_k)^{{\tilde X}_k-1}
{\mathbb I}_{\{{\tilde X}_k\ge 2\}}
\right),
\\
\eta_{k,4}
&=
\E\left(
\PP\{{\cal B}_4|X_k){\mathbb I}_{\{{\tilde X}_k\ge 2\}}
\right)
=
\E\left(
\frac{({\tilde X}_k)_2}{(n)_2}
(1-Q_k)^{2{\tilde X}_k-3}
{\mathbb I}_{\{{\tilde X}_k\ge 2\}}\right).
\end{align*}
We invoke these expressions in the identity 
$q_k=
\PP\{{\tilde X}_k\le 1\}
+ \eta_{k,1}+\cdots+\eta_{k,4}$. After a simple calculation
we obtain that
\begin{equation}
\nonumber
q_k
=
1
-
\frac{2}{n}{\tilde \kappa}
+
\Delta,
\qquad
\Delta
:=
\E
\left(
\frac{({\tilde X}_k)_2}{(n)_2}
\Bigl(1-2(1-Q_k)^{({\tilde X}_k-1)_+}
+
(1-Q_k)^{(2{\tilde X}_k-3)_+}
\Bigr)
\right).
\end{equation}
Furthermore, the inequality 
$(1-Q_k)^{({\tilde X}_k-1)_+}\ge (1-Q_k)^{(2{\tilde X}_k-3)_+}$ 
implies $\Delta\le \Delta'$, where 
\begin{displaymath}
\Delta'
=
\E
\left(
\frac{({\tilde X}_k)_2}{(n)_2}
h({\tilde X}_k,Q_k)
\right).
\end{displaymath}
Hence, we have 
\begin{equation}
\label{2023-05-18}
q_k\le 1-2n^{-1}{\tilde \kappa}+\Delta'.
\end{equation}
 Invoking this inequality in (\ref{2023-04-28+5}) and using $1+t\le e^t$ we obtain
\begin{displaymath}
\PP\{{\cal I}_{u,v}\}
\le 
\left(1-\frac{2}{n}{\tilde \kappa}+\Delta'\right)^m
\le
e^{-2\frac{m}{n}{\tilde \kappa}+m\Delta'}.
\end{displaymath}

Now we are ready to prove (\ref{2023-04-29+11}). Combining 
(\ref{2023-04-29+10}) with the 
 bound for  $\PP\{{\cal I}_{u,v}\}$ above we obtain
\begin{align*}
\Cov ({\mathbb I}_{{\cal I}_u},{\mathbb I}_{{\cal I}_v})
\le
e^{2m\ln\left(1-\frac{{\tilde \kappa}}{n}\right)
}
\left(
e^{-2\frac{m}{n}{\tilde \kappa}+m\Delta'
-
2m\ln\left(1-\frac{{\tilde \kappa}}{n}\right)}-1
\right).
\end{align*}
In view of (\ref{2023-04-28}) it suffices to show that 
$H:=-2\frac{m}{n}{\tilde \kappa}+m\Delta'
-
2m\ln\left(1-\frac{{\tilde \kappa}}{n}\right)$
($=$ the argument of the second exponent) is $o(1)$.
Invoking $\ln(1+t)=t+O(t^2)$ we obtain
$H=m\Delta'+O(mn^{-2})$. Recall that  $m=O(n\ln n)$.  Consequently,
the second term $O(mn^{-2})=o(1)$.  It remains to show that $m\Delta'=o(1)$ or, equivalently,  $\Delta'\, n\ln n=o(1)$.
 To this aim we split
\begin{align*}
\Delta'
&
=
\E\left(
 \frac{(X_k)_2}{(n)_2}
h(X_k,Q_k)
 {\mathbb I}_{\{X_k\le n\}}
 \right)
+
\E
\bigl(
h(n,Q_k)
{\mathbb I}_{\{X_k\ge n\}}
\bigr)
\\
&
=
\Delta'_1+\Delta'_2
\end{align*}
and consider $\Delta'_1$ and $\Delta'_2$ separately.
Using the inequality $h(n,Q_k)\le h(X_k,Q_k)$ for $X_k\ge n$ we upperbound $\Delta'_2$, 
\begin{displaymath}
\Delta'_2
\le 
\frac{1}{n\ln(1+n)}\E\Bigl(X_k
h(X_k,Q_k)
{\mathbb I}_{\{X_k\ge n\}}
\ln(1+ X_k)
\Bigr)
=
o\left(\frac{1}{n\ln n}\right).
\end{displaymath}
In the last step we use condition (\ref{2023-04-29+12}).
To show the correspoding bound for $\Delta'_1$ we make use of the
properties of the  function $\phi(x)$. Namely, for
$1\le x\le n$ we have 
\begin{displaymath}
x^2=x \phi(x)\ln(1+x) \frac{x}{\phi(x)\ln(1+x)}
\le
x \phi(x)\ln(1+x)\frac{n}{\phi(n)\ln(1+n)}.
\end{displaymath}
 Using this inequality we estimate
\begin{align*}
\Delta'_1
&
\le 
\E
\left(
\frac{X_k\phi(X_k)\ln(1+X_k)}{n\phi(n)\ln (1+n)}
h(X_k,Q_k){\mathbb I}_{\{X_k\le n\}}\right)
=
O
\left(
\frac{1}{n\phi(n)\ln(1+n)}
\right).
\end{align*}
In the last step we used (\ref{2023-06-13}).
Hence $\Delta'_1n\ln n =O(1/\phi(n))=o(1)$. 
The proof of  (\ref{2023-04-28+1}) is complete. We  arrive to (\ref{2023-04-27}).

\medskip

We secondly consider the case where
 ${\tilde\lambda}_{n,m}\to-\infty$ and show that
 \begin{equation}
\label{2023-05-01+3}
1-\PP\{G_{[n,m]}\ {\text{\rm{is connected}}}\}\to 0.
\end{equation} 
In the proof we use the following observation.

\begin{obs}\label{observation}
 Given  two sequences of positive integers
 $m''=m''(n)$ and $m'=m'(n)$  such that $m''>m'$
we can couple $G_{[n,m']}$ and $G_{[n,m'']}$
so that 
$\PP\{G_{[n,m']}\subset G_{[n,m'']}\}=1$. The coupling implies the inequality
$
\PP\{G_{n,m''}  \ {\text{is connected}}\}\ge 
\PP\{G_{n,m'} \ {\text{is connected}}\}$. Now we see that 
if 
(\ref{2023-05-01+3})
holds for 
$m=m'$ then (\ref{2023-05-01+3}) holds for $m=m''$ as well.
Hence to prove (\ref{2023-05-01+3}) it suffices to consider
sequences $m=m(n)$ satisfying ${\tilde\lambda}_{n,m}\to-\infty$
that grow slowly. For example, we may assume that 
\begin{equation}
 \label{2023-05-01+6}
 m\le \frac{2}{{\tilde{\kappa}}_n}n\ln n.
 \end{equation} 
\end{obs}

Let us prove  (\ref{2023-05-01+3}). $G_{[n,m]}=({\cal V},{\cal E})$ 
is disconnected if and only if there exists a vertex 
set $A\subset {\cal V}$ of size $1 \le |A|\le n/2$ such 
that $G_{[n,m]}$ contains no edges between $A$ and ${\cal V}\setminus A$.
 For every $A$  the probability that $G_{[n,m]}$ contains no
edges between $A$ and ${\cal V} \setminus A$, is  
${\bar q}_{|A|}^m$ (we use the fact that  $G_{n,1},\dots, G_{n,m}$ are independent and identically distributed).
Therefore, 
by the union bound,
\begin{equation}
\label{2023-05-13}
1
-
\PP\{G_{[n,m]}\ {\text{\rm{is connected}}}\}
\le
\sum_{1\le r\le n/2}{\binom{n}{r}}{\bar q}_r^m=:S.
\end{equation}
In what follows we show that $S=o(1)$. 

Denote, for short, 
$\alpha=\E\bigl(Q{\mathbb I}_{\{X\ge 2\}}\bigr)$.
Given $0<\beta<1$ we split  
\begin{displaymath}
S=S_1+S_2,
\qquad
S_1= 
\sum_{1\le r\le n^{\beta}}{\binom{n}{r}}{\bar q}_r^m,
\qquad
S_2=
\sum_{n^{\beta}< r\le n/2}{\binom{n}{r}}{\bar q}_r^m.
\end{displaymath}
For any $0<\beta<1$ we  show that $S_1=o(1)$. We also show that   for $\beta=\beta_0=\max\{1-\frac{\alpha}{{\color{red}2}\kappa}, \frac{1}{2}\}$ we have  $S_2=o(1)$. These two bounds imply $S=o(1)$.

\medskip

Proof of $S_2=o(1)$.  
Using (\ref{2023-05-03}) we upperbound  
\begin{equation}
\label{2023-05-12+1}
{\bar q}_{r}
=
\PP\{X\le 1\}
+
\E \bigl(q_r({\tilde X},Q){\mathbb I}_{\{{\tilde X}\ge 2\}}\bigr)
\le 
1-2\alpha \frac{r(n-r)}{n(n-1)}
\le 
1-{\color{red}\alpha}\frac{r}{n}.
\end{equation}
We will assume that $n$ is large enough so that  
${\tilde\lambda}_{n,m}<0$. Now the inequality 
${\tilde \kappa}\le \kappa$
implies 
$m> \frac{n}{{\tilde \kappa}}\ln n\ge \frac{n}{\kappa}\ln n$.
Hence
\begin{equation}
\label{2023-05-05}
{\bar q}_r^m
\le  
\left(
1
-
{\color{red}\alpha}\frac{r}{n}
\right)^m
\le 
e^{-{\color{red}\alpha} r\frac{m}{n}}
\le 
e^{-\frac{{\color{red}\alpha}}{\kappa} r\ln n}.
\end{equation}
Next, using Stirling's approximation for $n!$, $(n-r)!$ and $r!$ we upperbound the binomial coefficient 
\begin{displaymath}
\binom{n}{r}
\le 
\frac{n^n}{r^r(n-r)^{n-r}}=e^{\xi_1+\xi_2},
\qquad
\xi_1
:=
n\ln\left(\frac{n}{n-r}\right),
\qquad
\xi_2
:=
r\ln\left(\frac{n-r}{r}\right).
\end{displaymath}
Now we consider $\xi_1$ and $\xi_2$. For $r\ge n^{\beta}$ we have
$\xi_2\le (1-\beta)r\ln n$. Furthermore, the inequalities $\frac{1}{1-t}\le 1+2t$ and $\ln(1+2t)\le 2t$, for $0<t\le 0.5$ imply
\begin{displaymath}
\xi_1
=
n\ln\left(\frac{1}{1-\frac{r}{n}}\right)
\le 2r.
\end{displaymath}
We conclude that 
$\binom{n}{r}\le e^{2r+(1-\beta)r\ln n}$. Combining this bound with 
(\ref{2023-05-05}) we obtain that
\begin{displaymath}
\binom{n}{r}{\bar q}_r^m
\le e^{2r+(1-\beta)r\ln n-\frac{{\color{red}\alpha}}{\kappa}r\ln n}.
\end{displaymath}
Choosing $\beta=\beta_0$ we obtain $\binom{n}{r}{\bar q}_r^m
\le e^{2r-\frac{\alpha}{{\color{red}2}\kappa}r\ln n}$.
This inequality implies $S_2=o(1)$.

\medskip

Proof of $S_1=o(1)$.  Let $0<\beta<1$. 
We show below that for any $0<\tau<1$
and  $1\le r\le n^{\beta}$
\begin{equation}
\label{2023-05-01+8}
\E \bigl(
R_2(r,{\tilde X})h({\tilde X},Q)
\bigr)
\le 
\frac{r}{n}\frac{1}{\ln(1+\frac{\tau n}{r})}
\varphi\left(\frac{\tau n}{r}\right)
+
\frac{r}{n}
\frac{\tau}{2}
\E \left({\tilde X}
h({\tilde X},Q)
\right).
\end{equation}
Now assume that (\ref{2023-05-01+8}) holds.
Invoking $\tau=\frac{1}{\kappa\ln n}$
we obtain
 for large $n$  
(say, $n>n_{\beta,\kappa}$)
that
\begin{equation}
\label{2023-05-01+9}
\E  \bigl(
R_2(r,{\tilde X})h({\tilde X},Q)
\bigr)
\le
\frac{r}{n}\frac{1}{\ln n}
\end{equation}
uniformly in $1\le r\le n^{\beta}$.
Moreover, for large $n$ 
(say, $n>n_{\beta}$) we have 
$R_1\le \frac{r}{n}\frac{1}{\ln n}$.
Combining these inequalities with  (\ref{2023-05-01+1})
we obtain 
\begin{equation}
{\bar q}_r
\le 1-\frac{r}{n}{\tilde \kappa}+R,
\qquad
{\text{where}}
\quad R\le \frac{r}{n}\frac{2}{\ln n}.
\end{equation}
Finally, using $\binom{n}{r}
\le
\frac{n^r}{r!}$ we estimate
\begin{align*}
\binom{n}{r}
{\bar q}_r^m
\le
\frac{n^r}{r!}
e^{
-r
\frac{m}{n}
\left(
{\tilde \kappa}
-
\frac{2}{\ln n}
\right)
}
=
\frac{1}{r!}
e^{
r
\left(
{\tilde\lambda}_{n,m}
+
2
\frac{m}{n\ln n}
\right)
}.
\end{align*} 
Since, by (\ref{2023-05-01+6}), the fraction 
$\frac{m}{n\ln n}$ is bounded from above by a constant, our assumption   
${\tilde\lambda}_{n,m}\to-\infty$ implies $S_1=o(1)$.

It remains to prove (\ref{2023-05-01+8}). 
We split
\begin{align*}
\E \bigl(
R_2(r,{\tilde X})h({\tilde X},Q)
\bigr)
&
=
\E 
\Bigl(
R_2(r,{\tilde X})h({\tilde X},Q)
{\mathbb I}_{\{\frac{r}{n}{\tilde X}\le \tau\}}
\Bigr)
+
\E 
\Bigl(
R_2(r,{\tilde X})h({\tilde X},Q)
{\mathbb I}_{\{\frac{r}{n}{\tilde X}> \tau\}}
\Bigr)
\\
&
=I_1+I_2
\end{align*}
and estimate $I_1$ and $I_2$ separately.
Inequality $e^{-t}\le 1$,  for $t\ge 0$,
 implies
$R_2(r,x)\le \frac{r}{n}x$. Hence
\begin{align*}
I_2
&
\le 
\frac{r}{n}
\E 
\bigl({\tilde X}
h({\tilde X},Q)
{\mathbb I}_{\{\frac{r}{n}{\tilde X}> \tau\}}
\bigr)
\le 
\frac{r}{n}
\E\left(
{\tilde X}
\frac{\ln(1+{\tilde X})}{\ln(1+\frac{\tau n}{r})}
h({\tilde X},Q){\mathbb I}_{\{\frac{r}{n}{\tilde X}> \tau\}}
\right)
\\
&
\le 
\frac{r}{n}\frac{1}{\ln (1+\frac{\tau n}{r})}
\varphi\left(\frac{\tau n}{r}\right).
\end{align*}
To upper bound $I_1$ we estimate $R_2\le r^2x^2/(2n^2)$
using  $e^{-t}\le 1-t+t^2/2$, for $t>0$. 
We have
\begin{align*}
I_1
\le 
\frac{r^2}{2n^2}
\E\left(
{\tilde X}^2
h({\tilde X},Q)
{\mathbb I}_{\{\frac{r}{n}{\tilde X}\le \tau\}}
\right)
\le 
\frac{r}{n}
\frac{\tau}{2}
\E \left({\tilde X}
h({\tilde X},Q)
\right).
\end{align*}
The proof of (\ref{2023-05-08+4}) is complete. 

{\it Proof of (\ref{A}).}
Note that 
${\tilde \kappa}\le \kappa$ implies   ${\tilde \lambda}_{n,m}\ge \lambda'_{n,m}$. Consequently, 
  $\lambda'_{n,m}\to+\infty$ 
implies ${\tilde \lambda}_{n,m}\to+\infty$ and
$\PP\{G_{[n,m]}\ {\text{\rm{is connected}}}\}=o(1)$
follows from (\ref{2023-05-08+4}). 

Now we consider the case where  $\lambda'_{n,m}\to-\infty$. Assume for a moment that
 (\ref{2023-05-01+6}) holds.
The integrability condition (\ref{2023-04-29+12}) implies
that
\begin{displaymath}
0
\le 
\kappa
-
{\tilde \kappa}
\le
\E
\left(
Xh(X,Q){\mathbb I}_{\{X>n\}}
\right)
\le
\frac{1}{\ln(1+n)}
\E \left(
Xh(X,Q){\mathbb I}_{\{X>n\}}\ln(1+X)\right)
=o\left(\frac{1}{\ln n}
\right).
\end{displaymath}
This bound combined with (\ref{2023-05-01+6})  yields
$
{\tilde \lambda}_{n,m}
-
\lambda'_{n,m}
=o(1)$.  Hence ${\tilde \lambda}_{n,m}\to-\infty$
and  (\ref{2023-05-08+4}) implies
$\PP\{G_{[n,m]}\ {\text{\rm{is connected}}}\}= 1-o(1)$.
In the final step of the proof we revoke condition (\ref{2023-05-01+6}) using the coupling argument described in  Observation \ref{observation} above.

\end{proof}

\medskip

In the remaining part of the section we prove Theorem \ref{theorem}. Before the proof we introduce some notation and collect auxiliary results.
We denote
\begin{align*}
&
\kappa_{n,k}=\E\bigl(X_{n,k}h(X_{n,k},Q_{n,k})\bigr),
\qquad
\alpha_{n,k}
=
\E\bigl(Q_{n,k}{\mathbb I}_{\{X_{n,k}\ge 2\}}\bigr),
\\
&
\alpha_{n}=\E\bigl(Q_{n,i_*}{\mathbb I}_{\{X_{n,i_*}\ge 2\}}\bigr),
\qquad
\qquad
\
\,
\alpha
=
\E\bigl(Q{\mathbb I}_{\{X\ge 2\}}\bigr)
\end{align*}
and  observe that
\begin{displaymath}
\sum_{k=1}^m\alpha_{n,k}=m\alpha_n,
\qquad
\sum_{k=1}^m\kappa_{n,k}=m\kappa_n.
\end{displaymath}
Furthermore, we observe that the weak convergence  
$P_{n,m}\to P$ together with the convergence of moments (\ref{2023-06-03}) implies that $\alpha_n\to \alpha$ 
and $\kappa_n\to\kappa$ as $n\to+\infty$. Moreover,
the sequence 
of random variables 
$\{X_{n,i_*}h(X_{n,i_*},Q_{n,i_*})\ln(1+X_{n,i_*}), \, n\ge 1\}$ 
is uniformly integrable. That is, we have
\begin{displaymath}
\lim_{t\to+\infty}
\sup_{n}
\E
\left( 
X_{n,i_*}h(X_{n,i_*},Q_{n,i_*})\ln(1+X_{n,i_*})
{\mathbb I}_{\{X_{n,i_*}\ge t\}}
\right)
=
0.
\end{displaymath}
 Using the uniform integrability property above we can construct 
 a positive increasing function $\phi(x)$ such that 
$\lim_{x\to+\infty}\phi(x)=+\infty$ and
the sequence 
of random variables 
\linebreak
$\{X_{n,i_*}h(X_{n,i_*},Q_{n,i_*})\phi(X_{n,i_*})\ln(1+X_{n,i_*}), \, n\ge 1\}$ 
is uniformly integrable. With some ambiguity of notation we denote 
\begin{displaymath}
\varphi(t):=
\sup_{n}
\E
\left( 
X_{n,i_*}h(X_{n,i*},Q_{n,i_*})\phi(X_{n,i_*})\ln(1+X_{n,i_*})
{\mathbb I}_{\{X_{n,i_*}\ge t\}}
\right).
\end{displaymath}
The  uniform integrability implies  $\varphi(t)\to 0$ as $t\to+\infty$.
Moreover,  one can choose $\phi(x)$ such that
 the function 
$x\to x/(\phi(x)\ln (1+x))$ 
was increasing for $x=1,2,\dots$. {\color{red}For Reader's convenience we give a proof of these facts in  Section 3 (Appendix)}.

{\color{red}In the proof we use the fact that  for $m=O(n\ln n)$ we have
$\max_{1\le k\le m} \kappa_{n,k}=o(n)$.
To show this bound we fix $k$ and estimate  $\kappa^2_{n,k}$.
Combining 
Cauchy-Schwarz
and inequality $h(X_{n,k},Q_{n,k})\le~1$ we estimate
\begin{align*}
\kappa_{n,k}^2
\le 
\E \bigl(X_{n,k}h(X_{n,k},Q_{n,k})\bigr)^2
\le 
\E \bigl(X_{n,k}^2h(X_{n,k},Q_{n,k})\bigr).
\end{align*}
Summing up  we obtain $\sum_{k\in [m]}\kappa_{n,k}^2
\le m\E (X^2_{n,i_*}h(X_{n,i_*},Q_{n,i_*}))$. 
Furthermore, we estimate
\begin{align}
\nonumber
&
\E\bigl(X_{n,i_*}^2h(X_{n,i_*},Q_{n,i_*})\bigr)
=
\E \left(X_{n,i_*}^2h(X_{n,i_*},Q_{n,i_*})
\bigl(
{\mathbb I}_{\{X_{n,i_*}\le \sqrt n\}}
+
{\mathbb I}_{\{X_{n,i_*}> \sqrt n\}}\bigr)
\right)
\\
\label{2023-06-13+1}
&
\qquad
\le
\sqrt{n}\kappa_n
+
\frac{n}{\ln (1+n)}
\E \left(X_{n,i_*}h(X_{n,i_*},Q_{n,i_*}){\mathbb I}_{\{X_{n,i_*}> \sqrt{n}\}}\ln(1+X_{n,i_*})\right)
\\
\nonumber
&
\qquad
=
o\left(\frac{n}{\ln n}\right).
\end{align}
To get the  second term  in (\ref{2023-06-13+1})
we used $x=\ln(1+x)\frac{x}{\ln(1+x)}\le \ln(1+x)\frac{n}{\ln(1+n)}$ for $x:=X_{n,i_*}\le n$. We conclude that 
\begin{equation}
\label{2023-10-25}
\max_{1\le k\le n}\kappa^2_{n,k}
\le
\sum_{k\in[m]}\kappa^2_{n,k}
\le 
m \E (X^2_{n,i_*}h(X_{n,i_*},Q_{n,i_*}))
=o\left(\frac{mn}{\ln n}\right)=o(n^2).
\end{equation}
}

\begin{proof}[Proof of Theorem \ref{theorem}]
 It follows from the convergence
 $\alpha_n\to \alpha$ and $\kappa_n\to \kappa$ as $n\to\infty$ that for sufficiently large $n$ we have
\begin{displaymath}
\alpha_n\ge \alpha/2,
\qquad
\kappa/2\le \kappa_n\le 2\kappa.
\end{displaymath}

We firstly consider the case where
 $\lambda_{n,m}\to+\infty$ and prove (\ref{2023-04-27}).                                     
 The proof  is similar to that of the 
 corresponding  statement in  (\ref{2023-05-08+4}), see  proof of Proposition \ref{proposition}.
The only difference being  that 
the factors $p_1,\dots, p_m$ of the product
in (\ref{2023-04-28+2}) are now all distinct likewise  the  factors 
$q_1,\dots, q_m$ of the product  in (\ref{2023-04-28+5}).
Although the proof remains much the same, 
some adjustments  need to be made.
In what follows we  merely comment on the changes in the proof. 

Let us evaluate the probability  
$\PP\{{\cal I}_u\}=\prod_{k\in[m]}p_k$ of (\ref{2023-04-28+2}),   
where $p_k$ are defined similarly as 
in the proof of Proposition \ref{proposition} above,  but with 
${\tilde X}_k$ replaced by 
$X_{n,k}$. In particular, {\color{red}identity} (\ref{2023-05-13+20}) now reads as follows
$
p_k
=
1-n^{-1}\kappa_{n,k}$.
We have 
\begin{align}
\label{2023-05-17+20}
\PP\{{\cal I}_u\}
=
\prod_{k\in [m]}
p_k
= 
e^{\sum_{k\in[m]}\ln\left (1-\frac{\kappa_{n,k}}{n}\right)}
=
e^{-\frac{m}{n}\kappa_n-R},
\end{align}
where $0\le R\le \frac{1}{n^2}\sum_{k\in[m]}\kappa_{n,k}^2$.
In the last step we used inequalities {\color{red}$-t-t^2\le \ln(1-t)\le -t$ valid for $0\le t\le 0.5$ and the fact that 
$\max_{1\le k\le n}\kappa_{n,k}=o(n)$, see (\ref{2023-10-25}). 
In addition, (\ref{2023-10-25}) implies                 
 $R=o(1)$}. Hence
\begin{equation}
\label{2023-05-18+1}
\PP\{{\cal I}_u\}
=
e^{-\frac{m}{n}\kappa_n+o(1)}.
\end{equation}
The latter bound combined with 
(\ref{2023-05-17+20}) yields  the analogue of (\ref{2023-04-28})
\begin{equation}
\label{2023-05-18+2}
\E N_0=n\PP\{{\cal I}_u\}=e^{\lambda_{n,m}+o(1)}.
\end{equation}

Next we evaluate the probability
$\PP\{{\cal I}_{u,v}\}=\prod_{k\in[m]}q_k$  of
(\ref{2023-04-28+5}), where $q_k$ are defined similarly as 
in the proof of Proposition \ref{proposition} above,  but with 
${\tilde X}_k$ replaced by 
$X_{n,k}$. In particular, inequality (\ref{2023-05-18}) now reads as follows
\begin{displaymath}
q_k\le 1-\frac{2}{n}\kappa_{n,k}+\Delta'_{n,k},
\qquad
{\text{where}}
\qquad
\Delta'_{n,k}
=
\frac{1}{(n)_2}\E\bigl((X_{n,k})_2h(X_{n,k},Q_{n,k})\bigr).
\end{displaymath}
Using $1+t\le e^t$ and denoting 
$\Delta'_{*}=\frac{1}{(n)_2}\E\bigl((X_{n,i_*})_2h(X_{n,i_*},Q_{n,i_*})\bigr)$ 
we upper bound the probability
\begin{displaymath}
\PP\{{\cal I}_{u,v}\}
=
\prod_{k\in[m]}q_k
\le
\prod_{k\in[m]}\left(1-\frac{2}{n}\kappa_{n,k}+\Delta'_{n,k}\right)
\le
e^{-\frac{2m}{n}\kappa_n+m\Delta'_{*}}.
\end{displaymath}
Furthermore, invoking the bound $\Delta'_{*}=o\left(\frac{1}{n\ln n}\right)$,
(which is shown similarly to the bound
 $\Delta'=o\left(\frac{1}{n\ln n}\right)$ of the proof of Proposition \ref{proposition}) we obtain $m\Delta'_{*}=o(1)$. We have shown the bound
$\PP\{{\cal I}_{u,v}\}\le e^{-\frac{2m}{n}\kappa_n+o(1)}$.
Combining this bound with (\ref{2023-05-18+1}) we obtain
\begin{displaymath}
\Cov({\mathbb I}_{{\cal I}_u},{\mathbb I}_{{\cal I}_v})
= 
\PP\{{\cal I}_{u,v}\}- 
\PP\{{\cal I}_u\}
\PP\{{\cal I}_v\}
\le 
e^{-\frac{2m}{n}\kappa_n}o(1).
\end{displaymath}
The latter bound combined with  (\ref{2023-05-18+2}) yields
{\color{red}(\ref{2023-04-29+11})}. 
The remaining steps of the proof of 
(\ref{2023-04-27})  are the same as those of the proof of  
Proposition \ref{proposition}. We omit them.

\medskip

We secondly consider the case where
 $\lambda_{n,m}\to-\infty$ and show (\ref{2023-05-01+3}).                                     
By the coupling argument of  Observation \ref{observation} 
(see proof of Proposition \ref{proposition}) we can assume that
\begin{equation}
 \label{2023-05-13+1}
 m\le \frac{2}{{\kappa}_n}n\ln n.
 \end{equation}
 The proof of (\ref{2023-05-01+3}) is similar to that of the 
 corresponding result of Proposition \ref{proposition}.
The only difference being  that 
in (\ref{2023-05-13}) we replace ${\bar q}_r^m$ by the product
$\prod_{1\le k\le m}{\bar q}_r^{[n,k]}$,
where ${\bar q}_r^{[n,k]}=\E q_r(X_{n,k},Q_{n,k}\bigr)$.
We have
\begin{displaymath}
1
-
\PP\{G_{[n,m]}\ {\text{\rm{is connected}}}\}
\le
\sum_{1\le r\le n/2}{\binom{n}{r}}\prod_{1\le k\le m}{\bar q}_r^{[n,k]}=:S.
\end{displaymath}
To show $S=o(1)$ we split $S=S_1+S_2$ as in the proof of
Proposition \ref{proposition} and we set the value of the parameter 
$\beta$ to $\beta_0=\min\{1-\frac{\alpha}{{\color{red}8}\kappa},\frac{1}{2}\}$.

Proof of  $S_2=o(1)$.
Combining (\ref{2023-05-12+1})  and inequality $1+t\le e^t$ 
we estimate for each $k$
\begin{equation}
\label{2023-05-12+2}
{\bar q}_{r}^{[n,k]}
\le 
1-{\color{red}\alpha_{n,k}}\frac{r}{n}
\le 
e^{-{\color{red}\alpha_{n,k}}\frac{r}{n}}
\end{equation}
and upper bound the product
\begin{align*}
\prod_{1\le k\le m}{\bar q}_r^{[n,k]}
\le
e^{-\frac{r}{n}\sum_{1\le k\le m}{\color{red}\alpha_{n,k}}}
=
e^{-r\frac{m}{n}{\color{red}\alpha_n}}
\le 
e^{-r\frac{m}{n}{\color{red}\frac{\alpha}{2}}}
\le
e^{-\frac{r}{{\color{red}4}}\frac{\alpha}{\kappa}\ln n}.
\end{align*} 
In the last step we used inequalities
$m> \frac{n}{{\kappa}_{n}}\ln n\ge \frac{n}{2\kappa}\ln n$
(the first  one  follows from the inequality
$\lambda_{n,m}<0$, the second one follows from
$\kappa_n\le 2\kappa$). 
The remaining steps of the proof of 
$S_2=o(1)$ are the same as those of the proof of  Proposition \ref{proposition}. We omit them.

Proof of $S_1=o(1)$. In the proof we use the inequality 
\begin{equation}
\label{2023-05-12+4}
\E \bigl(
R_2(r,X_{n,i_*})h(X_{n,i_*},Q_{n,i_*})
\bigr)
\le 
\frac{r}{n}\frac{1}{\ln(1+\frac{\tau n}{r})}
\varphi\left(\frac{\tau n}{r}\right)
+
\frac{r}{n}
\frac{\tau}{2}
\kappa_n,
\end{equation}
which holds  for any $0<\tau<1$ and $1\le r\le n$. 
The proof of (\ref{2023-05-12+4}) is the same as that of (\ref{2023-05-01+8}). {\color{red}For Reader's convenience we
prove  (\ref{2023-05-12+4}) in Section 3 (Appendix).} 

Let $0<\beta<1$.
Choosing
$\tau=\frac{1}{\kappa\ln n}$ we obtain for large $n$ that
uniformly in $1\le r\le n^{\beta}$
\begin{equation}
\label{2023-05-12+9}
\E  \bigl(
R_2(r,X_{n,i_*})h(X_{n,i_*},Q_{n,i_*})
\bigr)
\le
\frac{r}{n}\frac{2}{\ln n}.
\end{equation}
Now we are ready to bound the product 
$\prod_{1\le k\le m}{\bar q}_{r}^{[n,k]}$. 
For each $k$ we estimate, see
(\ref{2023-05-01+1}),
\begin{equation}
{\bar q}_{r}^{[n,k]}
\le 
1-\frac{r}{n}\kappa_{n,k}+R_1
+
\E\bigl( R_2(r,X_{n,k})h(X_{n,k},Q_{n,k})\bigr)
\end{equation}
and using $1+t\le e^t$ upper bound the product
\begin{align*}
\prod_{1\le k\le m}{\bar q}_r^{[n,k]}
&
\le
e^{-\frac{r}{n}\sum_{k=1}^m\kappa_{n,k}
+
mR_1
+
\sum_{k=1}^m\E\bigl( R_2(r,X_{n,k})h(X_{n,k},Q_{n,k})\bigr)}
\\
&
=
e^{-\frac{r}{n}m\kappa_n
+
mR_1
+
m\E\bigl( R_2(r,X_{n,i_*})h(X_{n,i_*},Q_{n,i_*})\bigr)}
\\
&
\le e^{-\frac{rm}{n}\kappa_n+3\frac{rm}{n\ln n}}.
\end{align*}
In the last step we invoked (\ref{2023-05-12+9}) and inequality $R_1\le \frac{r}{n\ln n}$, which holds for large $n$. 

Next, using 
$\binom{n}{r}\le \frac{n^r}{r!}$ we estimate
\begin{align*}
\binom{n}{r}\prod_{1\le k\le m}{\bar q}_r^{[n,k]}
\le 
\frac{n^r}{r!}e^{-\frac{rm}{n}\kappa_{n}+3\frac{rm}{n\ln n}}
=
\frac{1}{r!}e^{r\left(\lambda_{n,m}-\frac{3m}{n\ln n}\right)}.
\end{align*}
Finally, since our assumption (\ref{2023-05-13+1}) and inequality 
$\kappa_n\ge \kappa/2$ imply that  $\frac{m}{n\ln n}$ is bounded by a constant, the bound  $S_1=o(1)$ follows because of 
$\lambda_{n,m}\to-\infty$.

\end{proof}

\section{Appendix}
\begin{lem}
\label{lemma_appendix}
Let $Y,Y_1,Y_2,\dots$ be a sequence of random variables.
Assume that $\E |Y|<\infty$ and
$\E |Y_n|<\infty$ for $n=1,2,\dots$.
Let $n\to+\infty$.
Assume that $Y_n\to Y$ in distribution and $\E |Y_n|\to \E |Y|$. 
Then there exists a non-decreasing function $g:{\mathbb R}\to[1,+\infty)$ (dependening on the distributions of $Y$ and $Y_n$, $n=1,2,\dots$) such that $g(t)\to+\infty$ as $t\to+\infty$ and
\begin{equation}
\label{2023-10-23}
\lim_{t\to+\infty}\sup_n\E\left(|Y_n|g(Y_n){\mathbb I}_{\{|Y_n|>t\}}\right)=0.
\end{equation}
\end{lem}
We remark that one can replace $g$ in (\ref{2023-10-23}) by another   non decreasing function $\phi$  growing to infinity (perhaps slower than $g$)  such that
 the function $t\to t/(\phi(t)\ln (1+t))$ was increasing.
 We denote 
 \begin{displaymath}
\varphi(t):= \sup_n\E\left(|Y_n|\phi(Y_n){\mathbb I}_{\{|Y_n|>t\}}\right)
\end{displaymath}
 and observe that $\varphi(t)\to 0$ as $t\to+\infty$.

\begin{proof}[Proof of Lemma \ref{lemma_appendix}]
We first show that 
\begin{equation}
\label{2023-10-23+1}
\forall\varepsilon>0 
\
\exists
\,
T_{\varepsilon}>0:
\
t\ge T_{\varepsilon}
\
\Rightarrow
\
\sup_{n}\E\left(|Y_n|{\mathbb I}_{\{|Y_n|>t\}}\right)<\varepsilon.
\end{equation}
By the convergence in distribution $Y_n\to Y$ we have for any 
 continuity point of the distribution function of $Y$ (c.p. for short) $s\in {\mathbb R}$ 
 that
 \begin{equation}
 \label{A++}
 \lim_n \E\left(|Y_n|{\mathbb I}_{\{|Y_n|\le s\}}\right)
 =
 \E \left(|Y|{\mathbb I}_{\{|Y|\le s\}}\right).
 \end{equation}
 We fix $\varepsilon>0$ 
  in (\ref{2023-10-23+1}). In view of $\E |Y|<\infty$ we find a c.p. 
  $s_{\varepsilon}>0$ such that
   
\begin{displaymath}
\E\left(|Y|{\mathbb I}_{\{|Y|> s_{\varepsilon}\}}\right)
< 
\frac{\varepsilon}{3}.
\end{displaymath}
By (\ref{A++}) there exists $n_0=n_0(\varepsilon)$ such that
 \begin{displaymath}
n\ge n_0
\
\Rightarrow
\
\left| \E\left(|Y_n|{\mathbb I}_{\{|Y_n|\le s_{\varepsilon}\}}\right)
 -
 \E\left(|Y|{\mathbb I}_{\{|Y|\le s_{\varepsilon}\}}\right) 
 \right|
 <\frac{\varepsilon}{3}
 \end{displaymath}
 Furthermore, the convergence $\E |Y_n|\to \E |Y|$ implies that
 there exists $n_1=n_1(\varepsilon)$ such that
 \begin{displaymath}
 n\ge n_1
 \
 \Rightarrow
 \
\left|\E|Y_n|-\E |Y|\right|<\frac{\varepsilon}{3}.
\end{displaymath} 
  Consequently, for $n>n_2:=\max\{n_0,n_1\}$ we have
 \begin{align}
 \label{2023-10-23+3}
  \E\left(|Y_n|{\mathbb I}_{\{|Y_n|>s_{\varepsilon}\}}\right)
  &
  =
\left(  \E |Y_n|-\E |Y|
\right)
  +
 \left( \E\left(|Y|{\mathbb I}_{\{|Y|\le s_{\varepsilon}\}}\right)
-
\E\left(|Y_n|{\mathbb I}_{\{|Y_n|\le s_{\varepsilon}\}}\right)
\right)
\\
\nonumber
&
+  
\E\left(|Y|{\mathbb I}_{\{|Y|> s_{\varepsilon}\}}\right)
  \\
  \nonumber
  &
  <
  \frac{\varepsilon}{3}+\frac{\varepsilon}{3} +\frac{\varepsilon}{3}
  = \varepsilon.
  \end{align}
Finally, the integrability of $Y_1,\dots, Y_{n_2}$ implies that
there exists $t_{\varepsilon}>0$ such that
\begin{equation}
\label{2023-10-23+4}
\max_{1\le n\le n_2}
\E\left(
|Y_n|{\mathbb I}_{\{|Y_n|\ge t_{\varepsilon}\}}
\right)
<\varepsilon.
\end{equation}
Combining (\ref{2023-10-23+3}) and  (\ref{2023-10-23+4})  
we obtain 
(\ref{2023-10-23+1}) 
with $T_{\varepsilon}=\max\{s_{\varepsilon}, t_{\varepsilon}\}$.
 
 \bigskip
 
 Now we derive (\ref{2023-10-23}) from (\ref{2023-10-23+1}).
We have that  the function
 \begin{displaymath}
 t\to h(t):=
 \sup_{n}\E\left(|Y_n|{\mathbb I}_{\{|Y_n|>t\}}\right)
 \end{displaymath}
 is non-increasing and, by (\ref{2023-10-23+1}), $h(t)\to 0$
as $t\to+\infty$.
We can choose a sequence $\{a_i,\, i\ge 0\}$ with $a_i\in \{0,1\}$ having infinitely many $1$'s such that
\begin{equation}
\label{2023-10-23+6}
\sum _{i\ge 0}a_ih(i)<\infty,
\qquad
\lim_{k\to+\infty}h(k)\sum_{i=0}^{k}a_i=0.
\end{equation}
For convenience we assume that $a_0=1$.  Denote $g(j)=\sum_{i=0}^ja_i$ and define the function
\begin{displaymath}
t\to g(t)=\sum_{i\ge 0}g_i{\mathbb I}_{\{i\le t<i+1\}}
\end{displaymath}
 (i.e. for each $i\ge 0$ and $t\in[i,i+1)$ we have 
 $g(t)=g(i)$).

Given $r$  introduce the function 
$h_r(t)=\E\left (|Y_r|{\mathbb I}_{\{|Y_r|>t\}}\right)$.
For any integer $k>0$  we have
\begin{align*}
\E\left (|Y_r|g(Y_r){\mathbb I}_{\{|Y_r|>k\}}\right) 
&
=\sum_{i\ge k}g(i)(h_r(i)-h_r(i+1))
\\
&
=
\lim_{n\to\infty}
\sum_{i=k}^ng(i)(h_r(i)-h_r(i+1))
\\
&
=
g(k)h_r(k)
-
\lim_{n\to\infty}\left(g(n)h_r(n+1)
-\sum_{i=k}^{n-1}h_r(i+1)(g(i+1)-g(i))
\right).
\end{align*} 
In the last step we applied the 
summation by parts formula.
Note that
\begin{displaymath}
\lim_{n\to\infty}
\sum_{i=k}^{n-1}h_r(i+1)(g(i+1)-g(i))
=\sum_{i\ge k}h_r(i+1)a_{i+1}
\le
\sum_{i\ge k}h(i+1)a_{i+1}
\end{displaymath}
and $g(k)h_r(k)\le g(k)h(k)$. Furthermore,  $g(n)h_r(n+1)\le g(n)h(n)\to 0$ as $n\to+\infty$, by the second relation of (\ref{2023-10-23+6}).
Hence 
\begin{displaymath}
\E\left (|Y_r|g(Y_r){\mathbb I}_{\{|Y_r|>k\}}\right) 
\le 
 g(k)h(k)+\sum_{i\ge k}h(i+1)a_{i+1}.
\end{displaymath}
Since this bound holds uniformly in $r$ we conclude from (\ref{2023-10-23+6}) that  
\begin{displaymath}
\lim_{t\to+\infty}
\sup_r\E\left (|Y_r|g(Y_r){\mathbb I}_{\{|Y_r|>t\}}\right) 
=0.
\end{displaymath}
We note that $g$ is nondecreasing, $g(0)=1$ and $g(t)\to+\infty$ as $t\to+\infty$. 

\end{proof}

{\it Proof of} (\ref{2023-05-12+4}). Recall that 
$R_2(r,x)=e^{-\frac{r}{n}x}-1+\frac{r}{n}x$.
In the proof we use inequalities
\begin{displaymath}
 R_2(r,x)\le \frac{r}{n}x,
 \qquad
 R_2(r,x)
 \le
 \frac{r^2x^2}{2n^2},
 \qquad
 {\text{ for}}
 \qquad
  x\ge 0,
 \end{displaymath}
which follow from respective inequalities 
$e^{-t}\le 1$ and $e^{-t}\le 1-t+t^2/2$ valid for
 $t\ge 0$.
 
 Let us prove (\ref{2023-05-12+4}). We split
 \begin{align*}
 \E\left(R_2(r,X_{n,i_*})h(X_{n,i_*}, Q_{n,i_*})\right)
 &
 =
 \E\left(R_2(r,X_{n,i_*})h(X_{n,i_*}, Q_{n,i_*})
( {\mathbb I}_{\{\frac{r}{n}X_{n,i_*}\le\tau\}}
+
 {\mathbb I}_{\{\frac{r}{n}X_{n,i_*}>\tau\}})\right)
 \\
 &
 =:I_1+I_2.
 \end{align*}
 Using $R_2(r,x)\le \frac{r}{n}x$ we estimate
 \begin{align*}
 I_2
&
 \le 
\frac{r}{n}  
\E\left(X_{n,i_*}h(X_{n,i_*}, Q_{n,i_*})
 {\mathbb I}_{\{\frac{r}{n}X_{n,i_*}>\tau\}}\right)
\\
&
 \le
 \frac{r}{n}  
\E\left(X_{n,i_*}\frac{\ln(1+X_{n,i_*})}{\ln(1+\frac{\tau n}{r})}
h(X_{n,i_*}, Q_{n,i_*})
 {\mathbb I}_{\{\frac{r}{n}X_{n,i_*}>\tau\}}\right)
\\
&
 \le 
 \frac{r}{n}
 \frac{1}{\ln(1+\frac{\tau n}{r})}\varphi\left(\frac{\tau n}{r}\right).
 \end{align*}
 Using $R_2(r,x)\le \frac{r^2x^2}{2n^2}$ we estimate
 \begin{align*}
 I_1
 \le 
 \frac{r^2}{2n^2}
 \E
 \left(X_{n,i_*}^2h(X_{n,i_*},Q_{n,i_*})
 {\mathbb I}_{\{\frac{r}{n}X_{n,i_*}\le\tau\}}
 \right)
 \le
 \frac{r\tau}{2n}
 \E
 \left(X_{n,i_*}h(X_{n,i_*},Q_{n,i_*})
 \right).
 \end{align*}
 Combining these estimates we obtain (\ref{2023-05-12+4}).

\end{document}